\documentclass[runningheads]{llncs}
\usepackage{graphicx}
\usepackage{appendix}
\usepackage{amssymb}
\usepackage{empheq}
\usepackage[mathscr]{euscript}
\usepackage{stmaryrd}
\usepackage{mathabx}
\usepackage{dsfont}
\usepackage{graphicx}
\usepackage{multirow}
\usepackage{algorithm2e}
\usepackage[table]{xcolor}
\usepackage{color}
\usepackage{arydshln}
\usepackage{tikz}
\usetikzlibrary{3d}
\usetikzlibrary{decorations.pathreplacing,angles,quotes,patterns}
\usepackage{pgfplots}
\usepgfplotslibrary{colorbrewer}
\pgfplotsset{cycle list/Set1-4}
\usepackage{algorithm2e}
\usepackage{braket}
\usepackage[bb=boondox]{mathalfa}
\usepackage{subcaption}
\usepackage{booktabs}
\usepackage{xfrac}

\usepackage[colorlinks=true,breaklinks=true,bookmarks=true,urlcolor=blue,citecolor=blue,linkcolor=blue,bookmarksopen=false,draft=false]{hyperref}

\newcommand{\bbN}{\mathbb{N}}
\newcommand{\bbR}{\mathbb{R}}

\newcommand{\Conv}{\operatorname{Conv}}

\newcommand{\gr}{\operatorname{gr}}

\newcommand{\supp}{\operatorname{supp}}
\newcommand\defeq{\mathrel{\overset{\makebox[0pt]{\mbox{\normalfont\tiny\sffamily def}}}{=}}}

\newcommand{\NN}{\texttt{NN}}

\newcommand{\ReLu}{\texttt{ReLU}}

\usepackage{pifont}

\tikzstyle{yzx} = [
  x={(1.2*.9625cm, 1.2*.9625cm)},
  y={(1.2*2.5cm, 0cm)},
  z={(0cm, 1.2*3cm)},
]

\def\notename{Note.}
\def\note{\par\addvspace{17pt}\small\rmfamily
\trivlist\if!\notename!\item[]\else
\item[\hskip\labelsep
{\bfseries\notename}]\fi}

\begin{document}

\title{Strong mixed-integer programming formulations for trained neural networks}
\titlerunning{MIP formulations for trained neural networks}
%
\author{Ross Anderson\inst{1} \and Joey Huchette\inst{2} \and Christian Tjandraatmadja\inst{3} \and Juan Pablo Vielma\inst{4}}
\authorrunning{R. Anderson et al.}
%
\institute{
Google Research,~\email{rander@google.com} \and
Google Research,~\email{jhuchette@google.com} \and
Google Research,~\email{ctjandra@google.com} \and
MIT,~\email{jvielma@mit.edu}
}

\maketitle              

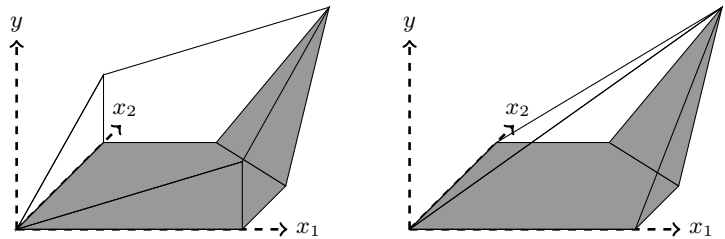
\begin{figure}
    \centering
    \begin{tikzpicture}[yzx]
        \draw [->, dashed, line width=1] (0,0,0) -- (1.2,0,0);
        \draw [->, dashed, line width=1] (0,0,0) -- (0,1.2,0);
        \draw [->, dashed, line width=1] (0,0,0) -- (0,0,0.7);
        \node[above right] at (1.2,-.075,0) {$x_2$};
        \node[right] at (0,1.2,0) {$x_1$};
        \node[above] at (0,0,.7) {$y$};
        \coordinate (LL) at (0,0,0);
        \coordinate (UL) at (1,0,0);
        \coordinate (LU) at (0,1,0);
        \coordinate (UU) at (1,1,0.5);
        \coordinate (UM) at (1,0.5,0);
        \coordinate (MU) at (0.5,1,0);
        \coordinate (E1) at (1,0,1/4);
        \coordinate (E2) at (0,1,1/4);

        \draw [fill=gray!80] (LL) -- (UL) -- (UM) -- (MU) -- (LU) -- cycle;
        \draw [fill=gray!80] (UU) -- (UM) -- (MU) -- cycle;
        \draw (LL) -- (E2) -- (UU) -- (E1) -- cycle;
        \draw (LL) -- (UL) -- (E1) -- cycle;
        \draw (LL) -- (LU) -- (E2) -- cycle;
        
    \end{tikzpicture} \hspace{2em}
    \begin{tikzpicture}[yzx]
        \draw [->, dashed, line width=1] (0,0,0) -- (1.2,0,0);
        \draw [->, dashed, line width=1] (0,0,0) -- (0,1.2,0);
        \draw [->, dashed, line width=1] (0,0,0) -- (0,0,0.7);
        \node[above right] at (1.2,-0.075,0) {$x_2$};
        \node[right] at (0,1.2,0) {$x_1$};
        \node[above] at (0,0,.7) {$y$};
        \coordinate (LL) at (0,0,0);
        \coordinate (UL) at (1,0,0);
        \coordinate (LU) at (0,1,0);
        \coordinate (UU) at (1,1,0.5);
        \coordinate (UM) at (1,0.5,0);
        \coordinate (MU) at (0.5,1,0);
        
        \draw [fill=gray!80] (LL) -- (UL) -- (UM) -- (MU) -- (LU) -- cycle;
        \draw [fill=gray!80] (UU) -- (UM) -- (MU) -- cycle;
        \draw (LL) -- (UL) -- (UU) -- cycle;
        \draw (LL) -- (LU) -- (UU) -- cycle;
    \end{tikzpicture}
    \caption{The convex relaxation for a ReLU neuron using: \textbf{(Left)} existing MIP formulations, and \textbf{(Right)} the formulations presented in this paper.}
    \label{fig:relu}
\end{figure}

\vspace{-0.9em}
\begin{abstract}
We present an ideal mixed-integer programming (MIP) formulation for a rectified linear unit (ReLU) appearing in a trained neural network. Our formulation requires a single binary variable and no additional continuous variables beyond the input and output variables of the ReLU. We contrast it with an ideal ``extended'' formulation with a linear number of additional continuous variables, derived through standard techniques. An apparent drawback of our formulation is that it requires an exponential number of inequality constraints, but we provide a routine to separate the inequalities in linear time. We also prove that these exponentially-many constraints are facet-defining under mild conditions. Finally, we study network verification problems and observe that dynamically separating from the exponential inequalities 1) is much more computationally efficient and scalable than the extended formulation, 2) decreases the solve time of a state-of-the-art MIP solver by a factor of 7 on smaller instances, and  3) nearly matches the dual bounds of a state-of-the-art MIP solver on harder instances, after just a few rounds of separation and in orders of magnitude less time.
\end{abstract}

\setcounter{footnote}{0}


\section{Introduction} 

Deep learning has proven immensely powerful at solving a number of important predictive tasks arising in areas such as image classification, speech recognition, machine translation, and robotics and control~\cite{Goodfellow:2016,LeCun:2015}. The workhorse model in deep learning is the feedforward network $\NN{}:\mathbb{R}^{m_0}\to \mathbb{R}^{m_s}$ with rectified linear unit (ReLU) activation functions, for which $\NN{}(x^0) = x^s$ is defined through 
\begin{equation} \label{eqn:simple-feedforward-network}
    x^i_j = \ReLu{}(w^{i,j} \cdot x^{i-1} + b^{i,j})
\end{equation}
for each layer $i \in \llbracket s \rrbracket \defeq \{1,\ldots,s\}$ and $j \in \llbracket m_i \rrbracket$. Note that the input $x^0 \in \bbR^{m_0}$ might be high-dimensional, and that the output $x^s \in \bbR^{m_s}$ may be multivariate. In this recursive description, $\ReLu{}(v) \defeq \max\{0,v\}$ is the ReLU activation function, and $w^{i,j}$ and $b^{i,j}$ are the weights and bias of an affine function which are learned during the training procedure. Each equation in \eqref{eqn:simple-feedforward-network} corresponds to a single \emph{neuron} in the network. Networks with any specialized linear transformations such as convolutional layers can be reduced to this model after training, without loss of generality.

There are numerous contexts in which one may want to solve an optimization problem containing a trained neural network such as $\NN{}$. For example, such problems arise in deep reinforcement learning problems with high dimensional action spaces and where any of the cost-to-go function, immediate cost, or the state transition functions are learned by a neural network~\cite{Arulkumaran:2017,Dulac-Arnold:2015,Mladenov:2017,Say:2017,Wu:2017}. Alternatively, there has been significant recent interest in verifying the robustness of trained neural networks deployed in systems like self-driving cars that are incredibly sensitive to unexpected behavior from the machine learning model~\cite{Carlini:2016,Papernot:2016,Szegedy:2013}. Relatedly, a string of recent work has used optimization over neural networks trained for visual perception tasks to \emph{generate} new images which are ``most representative'' for a given class~\cite{Olah:2017}, are ``dreamlike''~\cite{Mordvintsev:2015}, or adhere to a particular artistic style via neural style transfer~\cite{Gatys:2015}.
\subsection{MIP formulation preliminaries}

In this work, we study mixed-integer programming (MIP) approaches for optimization problems containing trained neural networks. In contrast to heuristic or local search methods often deployed for the applications mentioned above, MIP offers a framework for producing provably optimal solutions. This is of particular interest in the verification problem, where rigorous dual bounds can guarantee robustness in a way that purely primal methods cannot.

We focus on constructing MIP formulations for the \emph{graph} of ReLU neurons:
\begin{equation} \label{eqn:single-neuron}
\gr(\ReLu{} \circ f; [L,U]) \defeq \Set{ \left(x,(\ReLu{} \circ f) (x)\right) | L \leq x \leq U },
\end{equation}
where $\circ$ is the standard function composition operator $(g \circ f)(x) = g(f(x))$.
This substructure consists of a single ReLU activation function, taking as input an affine function $f(x) = w \cdot x + b$ over a $\eta$-dimensional box-constrained input domain. The nonlinearity is handled by introducing an auxiliary binary variable $z$ to indicate whether $(\ReLu{} \circ f) (x) = 0$ or $(\ReLu{} \circ f) (x) = f(x)$ for a given value of $x$. We focus on these particular substructures because we can readily produce a MIP formulation for the entire network as the composition of formulations for each individual neuron.\footnote{Further analysis of the interactions between neurons can be found in the full-length version of this extended abstract \cite{Anderson:2018a}.}

A MIP formulation is \emph{ideal} if the extreme points of its linear programming (LP) relaxation are integral. Ideal formulations are highly desirable from a computational perspective, and offer the strongest possible convex relaxation for the set being formulated~\cite{Vielma:2015}.

Our main contribution is an ideal formulation for a single ReLU neuron with no auxiliary continuous variables and an exponential number of inequality constraints. We show that each of these exponentially-many constraints is facet-defining under very mild conditions. We also provide a simple linear-time separation routine to generate the most violated inequality from the exponential family. This formulation is derived by constructing an ideal extended formulation that uses $\eta$ auxiliary continuous variables and projecting them out. We evaluate our methods computationally on verification problems for image classification networks trained on the MNIST digit dataset, where we observe that separating over these exponentially-many inequalities solves smaller instances faster than using Gurobi's default cut generation by a factor of 7, and (nearly) matches the dual bounds on larger instances in orders of magnitude less time.

\subsection{Relevant prior work}\label{sec:lit-review}

In recent years a number of authors have used MIP formulations to model trained neural networks~\cite{Bunel:2017,Cheng:2017,Dutta:2017,Fischetti:2018,Huchette:2018,Lomuscio:2017,Say:2017,Serra:2018a,Serra:2018,Tjeng:2017,Wu:2017,Xiao:2018}, mostly applying big-$M$ formulation techniques to ReLU-based networks. When applied to a single neuron of the form \eqref{eqn:single-neuron}, these big-$M$ formulations will not be ideal or offer an exact convex relaxation; see Example~\ref{ex:relu-big-M} for an illustration. Additionally, a stream of literature in the deep learning community has studied convex relaxations in the original space of input/output variables $x$ and $y$ (or a dual representation thereof), primarily for verification tasks~\cite{Bastani:2016,Dvijotham:2018a,Ehlers:2017}. It has been shown that these convex relaxations are equivalent to those provided by the standard big-$M$ MIP formulation, after projecting out the auxiliary binary variables (e.g.~\cite{Serra:2018a}). Moreover, some authors have investigated how to use convex relaxations within the training procedure in the hopes of producing neural networks with a priori robustness guarantees~\cite{Dvijotham:2018,Wong:2017,Wong:2018}. 

Beyond MIP and convex relaxations, a number of authors have investigated other algorithmic techniques for modeling trained neural networks in optimization problems, drawing primarily from the satisfiability, constraint programming, and global optimization communities~\cite{Bartolini:2011,Bartolini:2012,Katz:2017,Lombardi:2016,Schweidtmann:2018}. Another intriguing direction studies restrictions to the space of models that may make the optimization problem over the network inputs simpler: for example, the classes of binarized~\cite{Khalil:2018} or input convex~\cite{Amos:2017} neural networks.

Broadly, our work fits into a growing body of research in prescriptive analytics and specifically the ``predict, then optimize'' framework, which considers how to embed trained machine learning models into optimization problems~\cite{Bertsimas:2014k,Biggs:2017,Deng:2018,Donti:2017,Elmachtoub:2017,Hertog:2016,Misic:2017}. Additionally, the formulations presented below have connections with existing structures studied in the MIP and constraint programming community like indicator variables and on/off constraints~\cite{Atamturk:2018,Belotti:2015,Bonami:2015,Hijazi:2011,Hijazi:2014}.

\subsection{Starting assumptions and notation} \label{sec:assumptions}
We will assume that $-\infty < L_i < U_i < \infty$ for each input component $i$. While a bounded input domain will make the formulations and analysis considerably more difficult than the unbounded setting (see \cite{Atamturk:2018} for a similar phenomenon), it ensures that standard MIP representability conditions are satisfied (e.g. \cite[Section 11]{Vielma:2015}). Furthermore, variable bounds are natural for many applications (for example in verification problems), and are absolutely essential for ensuring reasonable dual bounds.

Define $\breve{L}, \breve{U} \in \bbR^\eta$ such that, for each $i \in \llbracket \eta \rrbracket$,
\[
    \breve{L}_i = \begin{cases} L_i & \text{ if } w_i \geq 0 \\ U_i & \text{ if } w_i < 0 \end{cases} \quad \text{ and } \quad \breve{U}_i = \begin{cases} U_i & \text{ if } w_i \geq 0 \\ L_i & \text{ if } w_i < 0 \end{cases}.
\]
This definition implies that $w_i \breve{L}_i \leq w_i \breve{U}_i$ for each $i$, which simplifies the handling of negative weights $w_i < 0$. Take the values $M^+(f) \defeq \max_{\tilde{x} \in [L,U]} f(\tilde{x}) \equiv w \cdot \breve{U} + b$ and $M^-(f) \defeq \min_{\tilde{x} \in [L,U]} f(\tilde{x}) \equiv w \cdot \breve{L} + b$. Define $\supp(w) \defeq \Set{ i \in \llbracket \eta \rrbracket | w_i \neq 0}$. Finally, take $\bbR_{\geq 0} \defeq \Set{x \in \bbR | x \geq 0}$ as the nonnegative orthant.

We say that \emph{strict activity} holds for a given ReLU neuron $\gr(\ReLu{} \circ f; [L,U])$ if $M^-(f) < 0 < M^+(f)$, or in other words, if $\gr(\ReLu{} \circ f; [L,U])$ is not equal to either $\gr(0; [L,U])$ or $\gr(f; [L,U])$. We assume for the remainder that strict activity holds for each ReLU neuron. This assumption is not onerous, as otherwise, the nonlinearity can be replaced by an affine function (either $0$ or $w \cdot x + b$). Moreover, strict activity can be verified or disproven in time linear in $\eta$.

\section{The ReLU neuron} \label{sec:relu}

The ReLU is the workhorse of deep learning models: it is easy to reason about, introduces little computational overhead, and despite its simple structure is nonetheless capable of articulating complex nonlinear relationships.

\subsection{A big-$M$ formulation} \label{ssec:big-M}

A standard big-$M$ formulation for $\gr(\ReLu{} \circ f; [L,U])$ is:
\begin{subequations} \label{eqn:relu-big-M}
\begin{align}
    y &\geq f(x) \label{eqn:relu-big-M-1} \\
    y &\leq f(x) - M^-(f) \cdot (1-z) \label{eqn:relu-big-M-2} \\
    y &\leq M^+(f) \cdot z \label{eqn:relu-big-M-3} \\
    (x,y,z) &\in [L,U] \times \bbR_{\geq 0} \times \{0,1\}.
\end{align}
\end{subequations}
This is the formulation used recently in the bevy of papers referenced in Section~\ref{sec:lit-review}. Unfortunately, this formulation is not necessarily ideal, as illustrated by the following example. 

\begin{example} \label{ex:relu-big-M}
    If $f(x) = x_1 + x_2 - 1.5$, formulation \eqref{eqn:relu-big-M} for $\gr(\ReLu{} \circ f; [0,1]^2)$ is
    \begin{subequations}
    \begin{align}
        y &\geq x_1 + x_2 - 1.5 \label{eqn:relu-example-1} \\ 
        y &\leq x_1 + x_2 - 1.5 + 1.5(1-z) \\
        y &\leq 0.5z \\
        (x,y,z) &\in [0,1]^2\times \bbR_{\geq 0} \times [0,1] \label{eqn:relu-example-5} \\
        z &\in \{0,1\}.
    \end{align}
    \end{subequations}
    The point $(\hat{x},\hat{y},\hat{z}) = ((1,0),0.25,0.5)$ is feasible for the LP relaxation (\ref{eqn:relu-example-1}-\ref{eqn:relu-example-5}); however, $(\hat{x},\hat{y}) \equiv ((1,0),0.25)$ is not in $\Conv(\gr(\ReLu{} \circ f; [0,1]^2))$, and so the formulation does not offer an exact convex relaxation (and, hence, is not ideal). See Figure~\ref{fig:relu} for an illustration: on the left, of the big-$M$ formulation projected to $(x,y)$-space, and on the right, the tightest possible convex relaxation.
\end{example}

The integrality gap of \eqref{eqn:relu-big-M} can be arbitrarily bad, even in fixed dimension $\eta$.

\begin{example}
Fix $\gamma \in \bbR_{\geq 0}$ and even $\eta \in \bbN$.
Take the affine function $f(x) = \sum_{i=1}^\eta x_i$, the input domain $[L,U] = [-\gamma,\gamma]^\eta$, and the point $\hat{x} = \gamma \cdot (1,-1,\cdots,1,-1)$ as a scaled vector of alternating ones and negative ones. We can check that $(\hat{x},\hat{y},\hat{z}) = (\hat{x},\frac{1}{2}\gamma\eta,\frac{1}{2})$ is feasible for the LP relaxation of the big-$M$ formulation \eqref{eqn:relu-big-M}. Additionally, $f(\hat{x}) = 0$, and for any $\tilde{y}$ such that $(\hat{x},\tilde{y}) \in \Conv(\gr(\ReLu{} \circ f; [L,U]))$, then $\tilde{y} = 0$ necessarily. Therefore, there exists a fixed point $\hat{x}$ in the input domain where the tightest possible convex relaxation (for example, from an ideal formulation) is exact, but the big-$M$ formulation deviates from this value by at least $\frac{1}{2}\gamma\eta$.
\end{example}

Intuitively, this example suggests that the big-$M$ formulation is particularly weak around the boundary of the input domain, as it cares only about the value $f(x)$ of the affine function, and not the particular input value $x$.

\subsection{An ideal extended formulation} \label{ssec:extended}

It is possible to produce an ideal \emph{extended} formulation for the ReLU neuron by introducing auxiliary continuous variables. The ``multiple choice'' formulation is
\begin{subequations} \label{eqn:relu-ideal-extended}
\begin{align}
    (x,y) &= (x^0,y^0) + (x^1,y^1) \label{eqn:relu-ideal-extended-1} \\
    y^0 &= 0 \geq w \cdot x^0 + b(1-z) \label{eqn:relu-ideal-extended-2} \\
    y^1 &= w \cdot x^1 + bz \geq 0 \label{eqn:relu-ideal-extended-3} \\
    L(1-z) &\leq x^0 \leq U(1-z)\label{eqn:relu-ideal-extended-4} \\
    Lz &\leq x^1 \leq Uz \label{eqn:relu-ideal-extended-5} \\
    z &\in \{0,1\},
\end{align}
\end{subequations}
is an ideal extended formulation for piecewise linear functions~\cite{Vielma:2009a}. It can alternatively be derived from techniques introduced by Balas~\cite{Balas:1985,Balas:1998}. Although the multiple choice formulation offers the tightest possible convex relaxation for a single neuron, it requires a copy $x^0$ of the input variables (note that it is straightforward to use equations \eqref{eqn:relu-ideal-extended-1} to eliminate the second copy $x^1$). This means that when the multiple choice formulation is applied to every neuron in the network to formulate $\NN{}$, the total number of continuous variables required is $m_0 + \sum_{i=1}^{r} (m_{i-1}+1)m_{i}$ (using the notation of \eqref{eqn:simple-feedforward-network}, where $m_i$ is the number of neurons in layer $i$). In contrast, the big-$M$ formulation requires only $m_0 + \sum_{i=1}^r m_i$ continuous variables to formulate the entire network. As we will see in Section~\ref{sec:computational_larger}, the quadratic growth in size of the extended formulation can quickly become burdensome. Additionally, a folklore observation in the MIP community is that multiple choice formulations tend to not perform as well as expected in simplex-based branch-and-bound algorithms, likely due to degeneracy introduced by the block structure~\cite{Vielma:2018}.

\subsection{An ideal non-extended formulation} \label{ssec:ideal-nonextended}
We now present a non-extended ideal formulation for the ReLU neuron, stated only in terms of the original variables $(x,y)$ and the single binary variable $z$. Put another way, it is the strongest possible tightening that can be applied to the big-$M$ formulation \eqref{eqn:relu-big-M}, and so matches the strength of the multiple choice formulation without the additional continuous variables. 

\begin{proposition} \label{prop:ideal-relu}
    Take some affine function $f(x) = w \cdot x + b$ over input domain $[L,U]$. The following is an ideal formulation for $\gr(\emph{\ReLu{}} \circ f; [L,U])$:
    \begin{subequations} \label{eqn:ideal-single-relu}
    \begin{align}
        y &\geq w \cdot x + b \label{eqn:ideal-single-relu-1} \\
        y &\leq \sum_{i \in I} w_i(x_i - \breve{L}_i(1-z)) + \left(b + \sum_{i \not\in I} w_i\breve{U}_i\right)z \quad \forall I \subseteq \supp(w) \label{eqn:ideal-single-relu-2} \\
        (x,y,z) &\in [L,U] \times \bbR_{\geq 0} \times \{0,1\} \label{eqn:ideal-single-relu-3}
    \end{align}
    \end{subequations}
\end{proposition}
\begin{proof}
See Appendix~\ref{app:ideal-relu}.
\qed \end{proof}

Furthermore, each of the exponentially-many inequalities in \eqref{eqn:ideal-single-relu-2} is necessary.

\begin{proposition} \label{prop:relu-facets}
    Each inequality in \eqref{eqn:ideal-single-relu-2} is facet-defining.
\end{proposition}
\begin{proof}
See Appendix~\ref{app:relu-facets}.
\qed \end{proof}

We require the assumption of strict activity above, as introduced in Section~\ref{sec:assumptions}. Under the same condition, it is also possible to show that~\eqref{eqn:ideal-single-relu-1} is facet-defining, but we omit it in this extended abstract for brevity. As a result of this and Proposition~\ref{prop:relu-facets}, the formulation~\eqref{eqn:ideal-single-relu} is minimal (modulo variable bounds).

The proof of Proposition~\ref{prop:relu-facets} offers a geometric interpretation of the facets induced by~\eqref{eqn:ideal-single-relu-2}. Each facet is a convex combination of two faces: an $(\eta - |I|)$-dimensional face consisting of all feasible points with $z = 0$ and $x_i = \breve{L}_i$ for all $i \in \llbracket \eta \rrbracket \setminus I$, and an $|I|$-dimensional face consisting of all feasible points with $z = 1$ and $x_i = \breve{U}_i$ for all $i \in I$.

It is also possible to separate from the family \eqref{eqn:ideal-single-relu-2} in time linear in $\eta$.

\begin{proposition} \label{prop:relu-separation}
    Take a point $(\hat{x},\hat{y},\hat{z}) \in [L,U] \times \bbR_{\geq 0} \times [0,1]$, along with the set
    \[
        \hat{I} = \Set{ i \in \supp(w) | w_i\hat{x}_i < w_i\left(\breve{L}(1-\hat{z}) + \breve{U}_i\hat{z}\right) }.
    \]
    If any constraint in the family \eqref{eqn:ideal-single-relu-2} is violated at $(\hat{x},\hat{y},\hat{z})$, then the one corresponding to $\hat{I}$ is the most violated.
\end{proposition}
\begin{proof}
    Follows from inspecting the family \eqref{eqn:ideal-single-relu-2}: each has the same left-hand-side, and so to maximize violation, it suffices to select the subset $I$ that minimizes the right-hand-side. This can be performed in a separable manner, independently for each component $i \in \supp(w)$, giving the result.
\qed \end{proof}

Observe that the inequalities \eqref{eqn:relu-big-M-2} and \eqref{eqn:relu-big-M-3} are equivalent to those in \eqref{eqn:ideal-single-relu-2} with $I = \supp(w)$ and $I = \emptyset$, respectively (modulo components $i$ with $w_i=0$). This suggests an iterative scheme to produce strong relaxations for ReLU neurons: start with the big-$M$ formulation \eqref{eqn:relu-big-M}, and use Proposition~\ref{prop:relu-separation} to separate strengthening inequalities from the exponential family \eqref{eqn:ideal-single-relu-2} as they are needed. We evaluate this approach in the following computational study.

\section{Computational experiments} \label{sec:computational}

To conclude the work, we study the strength of the ideal formulations presented in Section~\ref{sec:relu} for individual ReLU neurons. We study the verification problem on image classification networks trained on the canonical MNIST digit dataset~\cite{LeCun:1998}. We train a neural network $f : [0,1]^{28 \times 28} \to \bbR^{10}$, where the $10$ outputs correspond to the logits for each of the digits from 0 to 9. Given a labeled image $\tilde{x} \in [0,1]^{28 \times 28}$, our goal is to prove or disprove the existence of a perturbation of $\tilde{x}$ such that the neural network $f$ produces a wildly different classification result. If $f(\tilde{x})_i = \max_{j=1}^{10} f(\tilde{x})_j$, then image $\tilde{x}$ is placed in class $i$.
To evaluate robustness around $\tilde{x}$ with respect to class $j$, we can solve the following optimization problem for some small constant $\epsilon > 0$:
\[
    \max\nolimits_{a : ||a||_\infty \leq \epsilon}\: f(\tilde{x} + a)_j - f(\tilde{x} + a)_i.
\]
If the optimal solution (or a valid dual bound thereof) is less than zero, this verifies that our network is robust around $\tilde{x}$ in the sense that we cannot produce a small perturbation that will flip the classification from $i$ to $j$.

We train a smaller and a larger model, each with two convolutional layers with ReLU activation functions, feeding into a dense layer of ReLU neurons, and then a final dense linear layer. TensorFlow pseudocode specifying the two network architectures is included in Figure~\ref{fig:tf-pseudocode}. We generate 100 instances for each network by randomly selecting images $\tilde{x}$ with true label $i$ from the test data, along with a random target adversarial class $j \neq i$. Note that we make no attempts to utilize recent techniques that train the networks to be verifiable~\cite{Dvijotham:2018,Wong:2017,Wong:2018,Xiao:2018}.

\begin{figure}[htpb]
\begin{subfigure}[t]{1.0\textwidth}
{\smaller
\begin{verbatim}
input = placeholder(float32, shape=(28,28))
conv1 = conv2d(input, filters=4, kernel_size=4, 
              strides=(2,2), activation=relu, use_bias=True)
conv2 = conv2d(conv1, filters=4, kernel_size=4,
              strides=(2,2), use_bias=True)
flatten = reshape(conv2, [5*5*4])
dense = dense(flatten, 16, activation=relu, use_bias=True)
logits = dense(dense, 10, use_bias=True)
\end{verbatim}
}
\caption{Smaller ReLU network.}
\label{fig:small-relu}
\end{subfigure}
\begin{subfigure}[t]{1.0\textwidth}
\bigskip
{\smaller
\begin{verbatim}
input = placeholder(float32, shape=(28,28))
conv1 = conv2d(input, filters=16, kernel_size=4, 
              strides=(2,2), activation=relu, use_bias=True)
conv2 = conv2d(conv1, filters=32, kernel_size=4,
              strides=(2,2), activation=relu, use_bias=True)
flatten = reshape(conv2, [5*5*32])
dense = dense(flatten, 100, activation=relu, use_bias=True)
logits = dense(dense, 10, use_bias=True)
\end{verbatim}
}
\caption{Larger ReLU network.}
\label{fig:big-relu}
\end{subfigure}
\caption{TensorFlow pseudocode specifying the two network architectures used.}
\label{fig:tf-pseudocode}
\end{figure}

For all experiments, we use the Gurobi v7.5.2 solver, running with a single thread on a machine with 128 GB of RAM and 32 CPUs at 2.30 GHz. We use a time limit of 30 minutes (1800 s) for each run. We perform our experiments using the \texttt{tf.opt} package for optimization over trained neural networks; \texttt{tf.opt} is under active development at Google, with the intention to open source the project in the future. Below, the \emph{big-$M$ + \eqref{eqn:ideal-single-relu-2}} method is the big-$M$ formulation \eqref{eqn:relu-big-M} paired with separation\footnote{We use cut callbacks in Gurobi to inject separated inequalities into the cut loop. While this offers little control over when the separation procedure is run, it allows us to take advantage of Gurobi's sophisticated cut management implementation.} over the exponential family \eqref{eqn:ideal-single-relu-2}, and with Gurobi's cutting plane generation turned off. Similarly, the \emph{big-$M$} and the \emph{extended} methods are the big-$M$ formulation \eqref{eqn:relu-big-M} and the extended formulation \eqref{eqn:relu-ideal-extended} respectively, with default Gurobi settings. Finally, the \emph{big-$M$ + no cuts} method turns off Gurobi's cutting plane generation without adding separation over \eqref{eqn:ideal-single-relu-2}.

\subsection{Small ReLU network}
We start with a smaller ReLU network whose architecture is depicted in TensorFlow pseudocode in Figure~\ref{fig:small-relu}. The model attains 97.2\% test accuracy. We select a perturbation ball radius of $\epsilon = 0.1$. We report the results in Table~\ref{tab:small-relu} and in Figure~\ref{fig:perf-profile}. The big-$M$ + \eqref{eqn:ideal-single-relu-2} method solves $7$ times faster on average than the big-$M$ formulation. Indeed, for 79 out of 100 instances the big-$M$ method does not prove optimality after 30 minutes, and it is never the fastest choice (the ``win'' column). Moreover, the big-$M$ + no cuts times out on every instance, implying that using \emph{some} cuts is important. The extended method is roughly 5 times slower than the big-$M$ + \eqref{eqn:ideal-single-relu-2} method, but only exceeds the time limit on 19 instances, and so is substantially more reliable than the big-$M$ method for a network of this size. From this, we conclude that the additional strength offered by the ideal formulations \eqref{eqn:relu-ideal-extended} and \eqref{eqn:ideal-single-relu} can offer substantial computational improvement over the big-$M$ formulation \eqref{eqn:relu-big-M}.



\begin{table}[tb]
    \centering
    \begin{tabular}{lrrrr}
        \toprule
        method & time (s) & optimality gap & win \\ \midrule
        big-$M$ + \eqref{eqn:ideal-single-relu-2} &  174.49 & 0.53\% & 81 \\ 
        big-$M$  & 1233.49 & 6.03\% &  0 \\
        big-$M$ + no cuts & 1800.00 & 125.6\% & 0 \\
        extended &  890.21 & 1.26\% &  6 \\
        \bottomrule\\
    \end{tabular}
    \caption{Results for smaller network. Shifted geometric mean for time and optimality gap taken over 100 instances (shift of 10 and 1, respectively). The ``win'' column is the number of (solved) instances on which the method is the fastest.}
    \label{tab:small-relu}
\end{table}

\begin{figure}[htbp]
\begin{center}
\begin{tikzpicture}
\begin{axis}[
	xlabel=Time (s),
	ylabel=Number of instances solved,
	xmin=0,
	xmax=1800,
	ymin=0,
	ymax=100,
	cycle multi list={Set1-4},
	legend pos=outer north east,
	x tick label style={/pgf/number format/.cd, set thousands separator={}},
	scale=0.7
    ]

\addplot +[dotted, line width=1.5] coordinates {
(0.000,0)(200.016,0)(213.400,1)(250.701,2)(255.209,3)(389.366,4)(416.042,5)(428.185,6)(434.593,7)(477.248,8)(482.728,9)(549.913,10)(579.752,11)(626.431,12)(737.075,13)(872.758,14)(915.537,15)(958.458,16)(1012.978,17)(1147.197,18)(1256.312,19)(1573.421,20)(1800.000,20)
};

\addplot +[line width=1.5] coordinates { (0,0)
(0.000,0)(10.926,0)(11.029,1)(12.141,2)(12.574,3)(13.178,4)(13.310,5)(13.519,6)(14.511,7)(15.143,8)(16.212,9)(18.118,10)(19.818,11)(22.535,12)(22.964,13)(26.257,14)(30.373,15)(30.620,16)(33.884,17)(33.910,18)(35.863,19)(36.134,20)(44.185,21)(44.710,22)(46.004,23)(49.104,24)(52.136,25)(55.366,26)(56.669,27)(56.678,28)(62.419,29)(63.682,30)(64.232,31)(70.478,32)(75.827,33)(83.694,34)(88.192,35)(89.130,36)(93.818,37)(93.927,38)(96.258,39)(102.785,40)(108.975,41)(116.348,42)(117.292,43)(130.975,44)(135.703,45)(138.493,46)(140.969,47)(161.391,48)(162.063,49)(186.414,50)(186.478,51)(202.393,52)(213.807,53)(216.762,54)(221.753,55)(236.515,56)(237.472,57)(275.902,58)(304.160,59)(311.763,60)(466.913,61)(516.754,62)(564.719,63)(567.466,64)(596.327,65)(676.548,66)(711.947,67)(736.712,68)(783.253,69)(869.017,70)(895.416,71)(917.905,72)(940.710,73)(946.208,74)(955.712,75)(957.823,76)(1008.568,77)(1031.457,78)(1179.609,79)(1315.609,80)(1327.211,81)(1508.650,82)(1556.477,83)(1676.604,84)(1800.000,84)
};

\addplot +[dashed, line width=1.5] coordinates { (0,0)
(0.000,0)(291.123,0)(304.260,1)(337.765,2)(338.746,3)(354.233,4)(370.973,5)(413.280,6)(424.993,7)(460.116,8)(461.748,9)(463.096,10)(464.743,11)(477.083,12)(482.457,13)(483.141,14)(483.467,15)(500.758,16)(512.934,17)(523.865,18)(532.305,19)(546.597,20)(551.939,21)(560.813,22)(566.648,23)(567.355,24)(572.290,25)(625.472,26)(630.244,27)(638.643,28)(646.618,29)(649.133,30)(675.245,31)(678.859,32)(680.776,33)(691.489,34)(704.245,35)(707.126,36)(722.259,37)(725.898,38)(755.967,39)(775.662,40)(781.173,41)(791.785,42)(800.285,43)(808.156,44)(812.561,45)(819.229,46)(881.397,47)(887.791,48)(892.458,49)(918.828,50)(923.877,51)(937.224,52)(941.182,53)(949.978,54)(957.818,55)(965.734,56)(1038.186,57)(1087.226,58)(1113.039,59)(1116.502,60)(1139.682,61)(1139.921,62)(1154.051,63)(1202.194,64)(1211.171,65)(1211.259,66)(1218.455,67)(1222.273,68)(1244.772,69)(1273.453,70)(1327.791,71)(1329.087,72)(1361.641,73)(1384.576,74)(1389.731,75)(1427.386,76)(1448.842,77)(1471.646,78)(1516.231,79)(1517.456,80)(1800.000,80)
};


\addplot +[dashdotted, line width=1.5] coordinates {
(0.000,0)(1800,0)
};

\legend{big-$M$, big-$M$ + \eqref{eqn:ideal-single-relu-2}, extended, big-$M$ + no cuts}
\end{axis}
\end{tikzpicture}
\end{center}
\caption{Number of small network instances solved within a given amount of time. Curves to the upper left are better, with more instances solved in less time.}
\label{fig:perf-profile}
\end{figure}
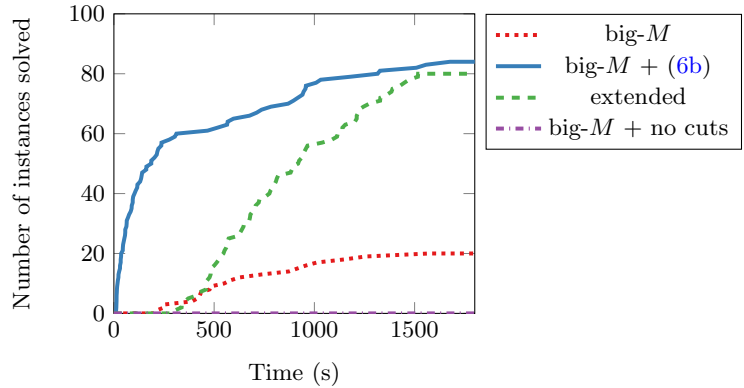

\subsection{Larger ReLU network} \label{sec:computational_larger}
Now we turn to the larger ReLU network described in Figure~\ref{fig:big-relu}. The trained model attains $98.5\%$ test accuracy. We select a perturbation ball radius of $\epsilon = 10 / 256$. For these larger networks, we eschew solving the problems to optimality and focus on the quality of the dual bound available at the root node. As Gurobi does not reliably produce feasible primal solutions for these larger instances, we turn off primal heuristics and compare the approaches based on the ``verification gap'', which measures how far the dual bound is from proving robustness (i.e. an objective value of 0). To evaluate the quality of a dual bound, we measure the ``improvement percentage'' $\frac{\texttt{big\_M\_bound} - \texttt{other\_bound}}{\texttt{big\_M\_bound}}$, where our baseline for comparison, \texttt{big\_M\_bound}, is the bound from the big-$M$ + no cuts method, and \texttt{other\_bound} is the dual bound being compared.

\begin{table}[tb]
    \centering
    \begin{tabular}{lrrr}
        \toprule
        method &  bound & time (s) & improvement \\ \midrule
        big-$M$ + no cuts    & 302.03 &    3.08 &     -   \\
        big-$M$ + \eqref{eqn:ideal-single-relu-2} & 254.95 &    8.13 & 15.44\% \\ 
        big-$M$  & 246.87 &  612.65 & 18.08\% \\
        big-$M$ + 15s timeout & 290.21 &   15.00 &  3.75\% \\
        extended &      - & 1800.00 & - \\
        \bottomrule\\
    \end{tabular}
    \caption{Results at the root node for larger network. Shifted geometric mean of bound, time, and improvement over 100 instances (shift of 10).}
    \label{tab:big-relu}
\end{table}

We report aggregated results over 100 instances in Table~\ref{tab:big-relu}. First, we are unable to solve even the LP relaxation of the extended method on any of the instances in the allotted 30 minutes, due to the quadratic growth in size. In contrast, the LP relaxation of the big-$M$ + no cuts method can be solved very quickly. The big-$M$ + \eqref{eqn:ideal-single-relu-2} method strengthens this LP bound by more than $15\%$ on average, and only takes roughly $2.5\times$ as long. This is not only because the separation runs very quickly, but also for a technical reason: when Gurobi's cutting planes are disabled, the callback separating over \eqref{eqn:ideal-single-relu-2} is only called a small number of times, as determined by Gurobi's internal cut selection procedure. Therefore, this 15\% improvement is the result of only a small number of separation rounds, not an exhaustive iterative procedure (i.e. Gurobi terminates the cut loop well before all violated inequalities have been separated).

We may compare these results against the big-$M$ method, which is able to provide a modestly better bound (roughly 18\% improvement), but requires almost two orders of magnitude more time to produce the bound. For another comparison, \emph{big-$M$ + 15s timeout}, we set a smaller time limit of 15 seconds on Gurobi, which is a tighter upper bound on the maximum time used by the big-$M$ + \eqref{eqn:ideal-single-relu-2} method. In this short amount of time, Gurobi is not able to improve the bound substantially, with less than 4\% improvement. This suggests that the inequalities \eqref{eqn:ideal-single-relu-2} are not trivial to infer by generic cutting plane methods, and that it takes Gurobi many rounds of cut generation to achieve the same level of bound improvement we derive from restricting ourselves to those cuts in \eqref{eqn:ideal-single-relu-2}.

\begin{acknowledgement}
 The authors gratefully acknowledge Yeesian Ng and Ond\v{r}ej S\'{y}kora for many discussions on the topic of this paper, and for their work on the development of the \texttt{tf.opt} package used in the computational experiments.
\end{acknowledgement}

%
%
%
\bibliographystyle{splncs04}
\bibliography{master.bib}

\newpage

\begin{appendix}
\section{Deferred proofs}
\subsection{Proof of Proposition~\ref{prop:ideal-relu}} \label{app:ideal-relu}
\begin{proof}
The result follows from applying Fourier--Motzkin elimination to \eqref{eqn:relu-ideal-extended} to project out the $x^0$, $x^1$, $y^0$, and $y^1$ variables;  see~\cite[Chapter 13]{Hooker:2000} for an explanation of the approach. We start by eliminating the $x^1$, $y^0$, and $y^1$ using the equations in \eqref{eqn:relu-ideal-extended-1}, \eqref{eqn:relu-ideal-extended-2}, and \eqref{eqn:relu-ideal-extended-3}, respectively, leaving only $x^0$.

First, if there is some input component $i$ with $w_i=0$, then $x^0_i$ only appears in the constraints (\ref{eqn:relu-ideal-extended-4}-\ref{eqn:relu-ideal-extended-5}), and so the elimination step produces $L_i \leq x_i \leq U_i$. 

Second, if there is some $i$ with $w_i < 0$, then we introduce an auxiliary variable $\tilde{x}_i$ with the equation $\tilde{x}_i = -x_i$. We then replace $w_i \leftarrow |w_i|$, $L_i \leftarrow -U_i$, and $U_i \leftarrow -L_i$, and proceed as follows under the assumption that $w > 0$.

Applying the Fourier-Motzkin procedure to eliminate $x^0_1$ gives the inequalities
\begin{align*}
    y &\geq w \cdot x + b \\
    y &\leq w_1x_1 - w_1L_1(1-z) + \sum_{i>1} w_ix^0_i + bz \\
    y &\leq w_1U_1z + \sum_{i>1} w_ix^0_i + bz \\
    y &\geq w_1x_1 - w_1U_1(1-z) + \sum_{i>1} w_ix^0_i + bz \\
    y &\geq w_1L_1z + \sum_{i>1} w_ix^0_i + bz \\
    L_1 &\leq x_1 \leq U_1,
\end{align*}
along with the existing inequalities in \eqref{eqn:relu-ideal-extended} where the $x^0_1$ coefficient is zero. Repeating this procedure for each remaining component of $x^0$ yields the linear system
\begin{subequations}
\begin{gather}
    y \geq w \cdot x + b \label{eqn:relu-nonnegative-1} \\
    y \leq \sum_{i \in I} w_i x_i - \sum_{i \in I} w_iL_i(1-z) + \left(b + \sum_{i \not\in I} w_iU_i\right)z \quad \forall I \subseteq \supp(w) \label{eqn:relu-nonnegative-2} \\
    y \geq \sum_{i \in I} w_i x_i - \sum_{i \in I} w_iU_i(1-z) + \left(b + \sum_{i \not\in I} w_iL_i\right)z \quad \forall I \subseteq \supp(w) \label{eqn:relu-nonnegative-3} \\
    (x,y,z) \in [L,U] \times \bbR_{\geq 0} \times [0,1]. \label{eqn:relu-nonnegative-4}
\end{gather}
\end{subequations}
Moreover, we can show that the family of inequalities \eqref{eqn:relu-nonnegative-3} is redundant, and can therefore be removed. Fix some $I \subseteq \supp(w)$, and take $h(I) \defeq \sum_{i \in I}w_i\breve{L}_i + \sum_{i \not\in I}w_i\breve{U}_i + b$. If $h(\llbracket \eta \rrbracket \setminus I) \geq 0$, we can express the inequality in \eqref{eqn:relu-nonnegative-3} corresponding to the set $I$ as a conic combination of the remaining constraints as:
\begin{align*}
y \geq w \cdot x + b && \times\quad & 1\\
0 \geq L_i - x_i && \times\quad & w_i & \forall i \notin I\\
0 \geq z - 1 && \times\quad & h(\llbracket \eta \rrbracket \setminus I)
\end{align*}

Alternatively, if $h(\llbracket \eta \rrbracket \setminus I) < 0$, we can express the inequality in \eqref{eqn:relu-nonnegative-3} corresponding to the set $I$ as a conic combination of the remaining constraints as:
\begin{align*}
y \geq 0 && \times\quad & 1\\
0 \geq x_i - U_i && \times\quad & w_i & \forall i \in I\\
0 \geq -z && \times\quad & -h(\llbracket \eta \rrbracket \setminus I)
\end{align*}

To complete the proof, for any components $i$ where we introduced an auxiliary variable $\tilde{x}_i$, we use the corresponding equation $\tilde{x}_i = -x_i$ to eliminate $x_i$ and replace it $\tilde{x}_i$, giving the result.
\qed \end{proof}

\subsection{Proof of Proposition~\ref{prop:relu-facets}} \label{app:relu-facets}
\begin{proof}
We fix $I = \{\kappa+1,\ldots,\eta\}$ for some $\kappa$; this is without loss of generality by permuting the rows of the matrices presented below. Additionally, we presume that $w \geq 0$, which allows us to infer that $\breve{L} = L$ and $\breve{U} = U$. This is also without loss of generality by appropriately interchanging $+$ and $-$ in the definition of the $\tilde{p}^k$ below. In the following, references to \eqref{eqn:ideal-single-relu-2} are taken to be references to the inequality in \eqref{eqn:ideal-single-relu-2} corresponding to the subset $I$.

Take the two points $p^0 = (x,y,z) = (L, 0, 0)$ and $p^1 = (U, f(U), 1)$. Each point is feasible with respect to \eqref{eqn:ideal-single-relu} and satisfies \eqref{eqn:ideal-single-relu-2} at equality. Then for some $\epsilon > 0$ and for each $i \in \llbracket \eta \rrbracket \backslash I$, take $\tilde{p}^i = (x,y,z) = (L + \epsilon {\bf e}^i, 0, 0)$. Similarly, for each $i \in I$, take $\tilde{p}^i = (x,y,z) = (U - \epsilon {\bf e}^i, f(U - \epsilon {\bf e}^i), 1)$. From the strict activity assumption, there exists some $\epsilon > 0$ sufficiently small such that each $\tilde{p}^k$ is feasible with respect to \eqref{eqn:ideal-single-relu} and satisfies \eqref{eqn:ideal-single-relu-2} at equality.

This leaves us with $\eta + 2$ feasible points satisfying \eqref{eqn:ideal-single-relu-2} at equality; the result then follows by showing that the points are affinely independent. Take the matrix
\[
    \begin{pmatrix} p^1 - p^0 \\ \tilde{p}^1 - p^0 \\ \vdots \\ \tilde{p}^\kappa - p^0 \\ \tilde{p}^{\kappa+1} - p^0 \\ \vdots \\ \tilde{p}^\eta - p^0 \end{pmatrix} = \begin{pmatrix} U-L & f(U) & 1 \\ \epsilon {\bf e}^1 & 0 & 0 \\ \vdots & \vdots & \vdots \\ \epsilon {\bf e}^\kappa & 0 & 0 \\ U - L - \epsilon {\bf e}^{\kappa+1} & f(U-\epsilon{\bf e}^{\kappa+1}) & 1 \\ \vdots & \vdots & \vdots \\ U - L - \epsilon {\bf e}^{\eta} & f(U-\epsilon{\bf e}^{\eta}) & 1 \end{pmatrix} \cong \begin{pmatrix} U - L & f(U) & 1 \\ \epsilon {\bf e}^1 & 0 & 0 \\ \vdots & \vdots & \vdots \\ \epsilon {\bf e}^\kappa & 0 & 0 \\ -\epsilon {\bf e}^{\kappa+1} & -w_{\kappa+1}\epsilon & 0 \\ \vdots & \vdots & \vdots \\ -\epsilon {\bf e}^{\eta} & -w_\eta \epsilon & 0 \end{pmatrix},
\]
where the third matrix is constructed by subtracting the first row to each of row $\kappa+2$ to $\eta+1$ (i.e. those corresponding to $\tilde{p}^i-p^0$ for $i > \kappa$), and is taken to mean congruency with respect to elementary row operations. If we permute the last column (corresponding to the $z$ variable) to the first column, we observe that the resulting matrix is upper triangular with a nonzero diagonal, and so has full row rank. Therefore, the starting matrix also has full row rank, as we only applied elementary row operations, and therefore the $\eta+2$ points are affinely independent, giving the result.
\qed \end{proof}

\end{appendix}

\end{document}